\documentclass[11pt]{article}
\usepackage[margin= 2 cm,bottom=19mm, top=18mm]{geometry}
\usepackage{amsmath,amssymb,amsthm,setspace,graphicx,array}

\usepackage[square,sort,comma,numbers]{natbib}
\usepackage{amsthm}
\usepackage{amsmath}
\usepackage{amssymb}
\usepackage{setspace}
\usepackage{mathtools}
\usepackage{graphicx}
\graphicspath{ {./images/} }
\usepackage[hidelinks]{hyperref}
\usepackage{cleveref}
\usepackage{hyperref}
\usepackage{enumitem}
\usepackage{framed}
\usepackage{subcaption}
\usepackage{xspace}
\usepackage{thmtools} 
\usepackage{thm-restate}

\usepackage{thmtools}
\usepackage{thm-restate}

\usepackage{floatrow}

\usepackage{tikz}
\usepackage{mathdots}
\usepackage{xcolor}
\usepackage{diagbox}
\usepackage{colortbl}
\usepackage[absolute,overlay]{textpos}

\newtheorem{theorem}{Theorem}[section]

\newtheorem{lemma}[theorem]{Lemma}

\newtheorem{definition}[theorem]{Definition}

\begin{document}
	
\title{Ramsey number of 1-subdivisions of transitive tournaments}
\author{
Nemanja Dragani\'c \thanks{
Department of Mathematics, ETH, Z\"urich, Switzerland. Research supported in part by SNSF grant 200021\_196965.
\emph{E-mails}: \textbf{\{nemanja.draganic,david.munhacanascorreia,benjamin.sudakov\}@math.ethz.ch}.
}
\and
David Munh\'a Correia\footnotemark[1]
\and
Benny Sudakov\footnotemark[1]
\and
Raphael Yuster \thanks{Department of Mathematics, University of Haifa, Haifa 3498838, Israel. Email: raphael.yuster@gmail.com\;.}
}
	
\date{}
	
\maketitle
	
\setcounter{page}{1}
	
\begin{abstract}
The study of problems concerning subdivisions of graphs has a rich history in extremal combinatorics. 
Confirming a conjecture of Burr and Erd\H{o}s, Alon proved in 1994 that subdivided graphs have linear Ramsey numbers. Later, Alon, Krivelevich and Sudakov showed that every $n$-vertex graph with at least $\varepsilon n^2$ edges contains a $1$-subdivision of the complete graph on $c_{\varepsilon}\sqrt{n}$ vertices, resolving another old conjecture of Erd\H{o}s. In this paper we consider the directed analogue of these problems and show that every tournament on at least $(2+o(1))k^2$ vertices contains the 1-subdivision of a transitive tournament on $k$ vertices. This
is optimal up to a multiplicative factor of 4 and confirms a conjecture of Gir\~ao, Popielarz and Snyder. 
	
\vspace*{3mm}

\end{abstract}

\section{Introduction}\label{sec:introduction}

Given a graph $G$, a \emph{subdivision of $G$} is a graph obtained by replacing its edges with internally vertex-disjoint paths of arbitrary length. More specifically, the \emph{$1$-subdivision of $G$} is the subdivision in which the length of these paths is $2$. Problems concerning subdivisions of graphs have been extensively studied in extremal combinatorics.

One of the central topics in discrete mathematics is the study of Ramsey numbers. The \emph{Ramsey number}, $r(G)$, of a graph $G$ is the smallest number $N$ such that every $2$-coloring of $K_N$ contains a monochromatic copy of $G$. A well known conjecture of Erd\H{o}s and Burr \cite{burr1975} was that subdivisions of graphs in which each subdivision path is of length at least 2, have Ramsey number which is linear in the number of vertices. Alon \cite{alon1994subdivided} resolved this in 1994, showing that every graph on $n$ vertices in which no two vertices of degree at least 3 are adjacent has Ramsey number at most $12n$. Later, Alon, Krivelevich and Sudakov \cite{alon2003turan} proved a stronger density-type result for cliques, showing that every $n$-vertex graph with at least $\varepsilon n^2$ edges contains the $1$-subdivision of a complete graph on $c_{\varepsilon}\sqrt{n}$ vertices. This proved an old conjecture of Erd\H{o}s \cite{erdos1979problems}.

In this paper, we study analogues of these problems in the framework of directed graphs. Notice that in this context it is only sensible to consider embedding \emph{acyclic} graphs in host digraphs, since in general the host digraph might not contain a directed cycle. Therefore, we will only consider subdivisions of the transitive tournament $T_k$ on $k$ vertices. Secondly, it is not possible to give a density-type statement as it was done in the result of Alon, Krivelevich and Sudakov \cite{alon2003turan}. Indeed, note that an orientation of the edges of the Tur\'an graph $T(n,k)$ with $k$ parts in which the direction of an edge between two parts conforms to a previously specified ordering of the parts, does not even contain a path of length $k$. Hence, only in very dense host directed graphs can we hope to embed an arbitrary subdivision of $T_k$, let alone the $1$-subdivision (Scott \cite{scott2000subdivisions}, in fact, proved that one can find a non-specified subdivision of $T_k$ inside of every $n$-vertex digraph with more edges than $T(n,k)$). This naturally leads to the following Ramsey-type question: How many vertices should a tournament have in order to contain the $1$-subdivision of $T_k$? 

The \emph{oriented Ramsey number}, $\overrightarrow{r}(H)$, of an oriented graph $H$ is the smallest number $N$ such that every tournament on $N$ vertices contains a copy of $H$. The study of oriented Ramsey numbers goes back 60 years to the work of Stearns, and Erd\H{o}s and Moser, who showed that $\overrightarrow{r}(T_k)$ is exponential in $k$ (see e.g., \cite{FHW} and its references for the history of this subject and some more recent results). The above question then asks for the oriented Ramsey number of the $1$-subdivision of $T_k$. This problem was raised by Gir\~ao, Popielarz and Snyder \cite{girao2021subdivisions}, who gave an upper bound of $O(k^2\log^3k)$. They also conjectured that 1-subdivisions of $T_k$ actually have linear oriented Ramsey number. We prove this conjecture.
\begin{theorem}\label{t:main}
Every tournament on $2k^2(1+o(1))$ vertices contains the $1$-subdivision of $T_k$.
\end{theorem}
\noindent The above result is optimal up to a factor of 4, since the $1$-subdivision of $T_k$ has at least $(1+o(1))k^2/2$ vertices. In the next section, we give some preliminaries and then prove the result in Section 3. We finish with some brief concluding remarks.

\section{Preliminary results and proof ideas}
We mainly use standard terminology. For a directed graph $G$ and a vertex $v \in V(G)$, let
$N^+(v)$ and $N^-(v)$ denote the set of out-neighbors and the set of in-neighbors of $v$ in $G$, respectively.
The out-degree of $v$ is $d^+(v)=|N^+(v)|$ and the in-degree of $v$ is $d^-(v)=|N^-(v)|$.
An edge from $u$ to $v$ in a directed graph is denoted by $(u,v)$ and an edge between $u$ and $v$ in 
an undirected graph is denoted by $uv$. Let $T_k$ be the transitive tournament on vertices $\{v_1,\ldots,v_k\}$
where $(v_i,v_j) \in E(T_k)$ for $1 \le i < j \le k$. Considering its $1$-subdivision $H_k$, we call $v_1,\ldots,v_k$
the {\em base vertices} and for every $1 \le i < j \le k$ there is a unique vertex $w_{i,j}$ such that
$(v_i,w_{i,j})$ and $(w_{i,j},v_j)$ are the only edges incident with $w_{i,j}$. We call $w_{i,j}$ the
{\em subdivision vertex} connecting $v_i$ to $v_j$. 

Given a pair of vertices in a tournament, it will be handy for us to quantify how well the pair is connected by directed paths of length two. This is captured in the following definition.

\begin{definition}\label{def:connections}
For two vertices $u,v$ of a directed graph, we define 
$$
c(u,v)= \max \{|N^+(u) \cap N^-(v)|\,,\,|N^+(v) \cap N^-(u)|\}\;.
$$
\end{definition}

Observe that if $u,v$ are vertices of a tournament and $d^+(u) \ge d^+(v)$, then clearly $c(u,v) \le |N^+(u) \cap N^-(v)|+1$. The following simple lemma shows that for every vertex in a tournament there always exists another vertex such that the pair they form is well connected in the sense of Definition~\ref{def:connections}.
\begin{lemma}\label{l:0}
	Let $T$ be a tournament on $n$ vertices. Then for every vertex $u \in V(T)$ there exists a vertex
	$v \in V(T)$ such that $c(u,v) \ge (n-3)/4$.
\end{lemma}
\begin{proof}
	Without loss of generality, assume that $d^+(u) \ge (n-1)/2$. Let $v$ be a vertex of minimum out-degree in $T[N^+(u)]$, i.e. the subtournament of $T$ induced by
	$N^+(u)$. Then $v$ has out-degree at most $(d^+(u)-1)/2$ in $T[N^+(u)]$.
	Thus, $|N^+(u) \cap N^-(v)| \ge (d^+(u)-1)-(d^+(u)-1)/2 \ge (n-3)/4$.
\end{proof}
\noindent
We now define, for a tournament $T$ and each $t \ge 1$, the undirected graph $T_{\le t}$ on the vertex set $V(T)$ to consist of those edges $uv$ such that $c(u,v) \le t$. The previous lemma trivially indicates that this graph must be sparse.
\begin{lemma}\label{l:1}
	The maximum degree of $T_{\le t}$ is at most $4t+2$.
\end{lemma}
\begin{proof}
	Consider a vertex $v$ and let $X$ be its set of neighbors in $T_{\le t}$.
	Then we must have $|X| \le 4t+2$ since otherwise the sub-tournament $T[X \cup \{v\}]$ has more than $4t+3$
	vertices and thus by Lemma \ref{l:0}, there exists a vertex $u \in X$ such that $c(u,v) > t$,
	contradicting the definition of $T_{\le t}$.
\end{proof}


Let us now outline the main ideas behind the proof of Theorem \ref{t:main}. Given a tournament $T$, we will take a random subset $A$ of vertices of $T$ of expected size slightly larger than $k$. We then show that in fact for every $t\leq |V(T)|$ we do not expect too many pairs $(u,v)$ in $A$ which have $c(u,v)<t$. This will allow us, after removing some vertices from $A$, to embed $H_k$ into $T$ by using the remaining vertices $A'\subseteq A$ as base vertices. We will employ a greedy embedding strategy by connecting the pairs in $A'$ one by one, giving priority to the pairs which have fewer possible connections. The next simple lemma describes the framework of our greedy embedding strategy, and is tailored for the use on the set $A'$ which we will be able to find.

\begin{lemma}\label{lem:greedy}
    Let $T$ be a tournament and let $A'=\{v_1,\ldots, v_k\}$ be a subset of its vertices such that we can order all pairs $e_1,e_2\ldots e_{\binom{k}{2}}$ contained in $A'$, so that for every $t\leq \binom{k}{2}$ for the pair $e_t=(v_i,v_j)$ (where $i<j$) it holds that $|N^+(v_i)\cap N^-(v_j)\setminus A'|\geq t$. Then $T$ contains $H_k$. 
\end{lemma}
\begin{proof}
We let $A'$ be the base set of the copy of $H_k$ which we want to find, and we greedily find the connections in $V(T)\setminus A'$ for each pair of vertices following the order given in the statement, and noting that by assumption there is at least one free vertex which we can use for the current pair.
\end{proof}

 \section{Randomised embedding of $H_k$}\label{sec:proof}
Throughout the rest of this section we assume, whenever necessary, that $k$ is sufficiently large.
Let $T$ be tournament with $K=2(k^2+k^{1.9})$ vertices. 
The following lemma shows that $T$ contains a set of vertices which we will later use in order to apply Lemma~\ref{lem:greedy} and complete the proof of 
Theorem \ref{t:main}.

\begin{lemma}\label{lem:random subset}
    $T$ contains a subset $A$ of vertices such that the following hold:
    \begin{itemize}
        \item[$({\rm P}1)$]  $|A| \ge k+ 2k^{0.8}$.
        \item[$({\rm P}2)$] Let $q_t$ be the number of distinct pairs $u,v$ in $A$ for which $c(u,v)<t$. For all integers $t$ with $k^{0.8} \le t \le K$ it holds that $q_t \le t-\frac{t}{32k^{0.1}}$.
\item[$({\rm P}3)$] For all pairs of distinct vertices $u,v \in V(T)$ with
$|N^+(u) \cap N^-(v)|=t \ge k^{0.7}$ it holds that $$|N^+(u) \cap N^-(v) \cap A| \le t/(2k) + t^{3/4}-1.$$
    \end{itemize}
\end{lemma}

\begin{proof}

Select a random subset $A \subseteq V(T)$ where each $v \in V(T)$ is independently chosen with probability
$p=\frac{1}{2(k+\frac{3}{4}k^{0.9})}.$ For each of the three listed properties in the statement, we show that each individual one holds with probability more than $2/3$ for the randomly chosen set $A$, thus completing the proof of the lemma. 

\textbf{Property $({\rm P}1)$.} Notice that $|A| \sim \text{Bin}(K,p)$ and so its expectation is $Kp$, which satisfies
$2k \ge Kp \ge k+k^{0.9}/8$, and its variance is $Kp(1-p) \le 2k$. Thus, by Chebyshev's inequality,
$$
\mathbb{P}\left(\left||A| - Kp \right| > \frac{k^{0.9}}{16}\right) < \frac{1}{3}\;.
$$
Hence the first property holds with probability more than $2/3$.



\textbf{Property $({\rm P}2)$.} Fix an integer $t$ with  with $k^{0.8} \le t \le K$. Notice that $q_t$ is the number of edges of
$T_{\le t}$ with both endpoints in $A$. Hence, $q_t = \sum_{uv \in E(T_{\le t})}X_{uv}$
where $X_{uv}$ is the indicator variable for the event that both endpoints $u,v$
are chosen to $A$. By the definition of $A$ we have that $\Pr[X_{uv}=1]=p^2$ and further,
by Lemma~\ref{l:1}, we know that $|E(T_{\le t})| \le (2t+1)K$. Thus,
\begin{align}
	\mathbb{E}[q_t] & \le  (2t+1)Kp^2 \nonumber
	= (2t+1)\frac{2(k^2+k^{1.9})}{4(k+\frac{3}{4}k^{0.9})^2} \nonumber
	 = \left(t + \frac{1}{2}\right)\frac{k^2+k^{1.9}}{k^2+\frac{3}{2}k^{1.9}+\frac{9}{16}k^{1.8}} \nonumber\\
	& \le \left(t + \frac{1}{2}\right)\left(1-\frac{\frac{1}{2}k^{1.9}}{k^2+\frac{3}{2}k^{1.9}}\right) \nonumber
		 \le \left(t + \frac{1}{2}\right)\left(1-\frac{1}{4k^{0.1}}\right) \nonumber
		 \le t + \frac{1}{2}-\frac{t}{4k^{0.1}} \nonumber
 		\le t -\frac{t}{8k^{0.1}} \;. 
	\end{align}
	We would now like to show that $q_t$ does not deviate much from its expected value, so we estimate its variance.
	As the choice of each vertex to $A$ is made independently, we have that $X_{uv}$ is independent of
	$X_{u'v'}$ whenever $uv$ and $u'v'$ are disjoint edges of $T_{\le t}$. Thus, by Lemma \ref{l:1},
	there are at most $|E(T_{\le t})|(8t+2)\le (2t+1)K(8t+2)$ ordered  pairs $uv,u'v'$ for which
	$X_{uv}$ and $X_{u'v'}$ are not independent. As for each non-independent pair we have $\Pr[X_{uv}=1 \wedge X_{u'v'}=1]=p^3$, we obtain that
	\begin{equation}\label{e:var}
		\text{Var}[q_t] \le \mathbb{E}[q_t] + (2t+1)(8t+2)Kp^3 \le t + 17t^2\cdot(3k^2)\cdot \frac{1}{8k^3}
			\le t + \frac{7t^2}{k}\;.
	\end{equation}
	Consider first the case where $t \le k/7$, for which we have by \eqref{e:var} that $\text{Var}[q_t] \le 2k/7$.
	By Chebyshev's inequality and by our estimate on the expectation of $q_t$ we then get
	\begin{equation}\label{e:var1}
	\mathbb{P}\left(q_t \ge  t-\frac{t}{16k^{0.1}}\right) \le 
	\mathbb{P}\left(q_t - \mathbb{E}[q_t] \ge  \frac{t}{16k^{0.1}}\right) \le \frac{2k/7}{t^2/256k^{0.2}}=O\left(\frac{k^{1.2}}{t^2}\right) = O\left(\frac{1}{k^{0.4}}\right)\;,
\end{equation}
where we are also using that $t \geq k^{0.8}$.
	Consider next the case where $t \ge k/7$, for which we have by \eqref{e:var} that $\text{Var}[q_t] \le 14t^2/k$.
	Now we have
	\begin{equation}\label{e:var2}
	\mathbb{P}\left(q_t \ge  t-\frac{t}{16k^{0.1}}\right) \le 
	\mathbb{P}\left(q_t - \mathbb{E}[q_t] \ge  \frac{t}{16k^{0.1}}\right) \le \frac{14t^2/k}{t^2/256k^{0.2}} =  O\left(\frac{1}{k^{0.8}}\right)\;.
\end{equation}
	As the number of possible choices for $t$ is $\Theta(k^2)$, we cannot just use \eqref{e:var1},  \eqref{e:var2} and the union bound to
	guarantee that $q_t \le  t-\frac{t}{16k^{0.1}}$ holds with high probability for all $t$.
	Instead, we proceed as follows. Let $S$ be the set of integers of the form $t_i = k^{0.8} \left( 1 + 1/32k^{0.1}\right)^i$ which are contained in $[k^{0.8},K]$ - clearly, $|S| \leq k^{0.2}$. We will prove that with probability larger than $2/3$, we have $q_t \le  t-\frac{t}{16k^{0.1}}$ for every $t \in S$. Once we show that, we are done since
	for each $t \in [k^{0,8},K]$, letting $t_i \in S$ be such that $t \le t_i \le t \left(1+1/32k^{0.1}\right)$, we have
	$$
	q_t \le q_{t_i} \le t_i-\frac{t_i}{16k^{0.1}}  \le t+\frac{t}{32k^{0.1}} - \frac{t}{16k^{0.1}} = t-\frac{t}{32k^{0.1}}
	$$
	establishing the lemma. Thus, it remains to apply \eqref{e:var1}, \eqref{e:var2} and the union bound
	to the elements of $S$. Indeed, this follows since $S$ has size at most $k^{0.2}$ and for each $t_i \in S$, $q_{t_i} \leq t_i-\frac{t_i}{16k^{0.1}}$ occurs with probability $O(1/k^{0.4})$. Hence, with probability larger than
	$2/3$, $q_t \le  t-\frac{t}{16k^{0.1}}$ for every $t \in S$.

\textbf{Property (${\rm P}3$).} Fix a pair of vertices $u,v \in V(T)$ for which $|N^+(u) \cap N^-(v)|=t \ge k^{0.7}$.
	Then $Z=|N^+(u) \cap N^-(v) \cap A| \sim \text{Bin}(t,p)$.
	By Chernoff's inequality, it holds then that
	$$
	\Pr\left[Z- tp > t^{2/3} \right] \le e^{-\frac{2t^{4/3}}{t}}  < \frac{1}{15k^4}\;.
	$$
	Hence, the probability that $|N^+(u) \cap N^-(v) \cap A|$ is larger than
	$$
	tp+t^{2/3} \le \frac{t}{2k}+t^{2/3} \le \frac{t}{2k}+t^{3/4}-1
	$$
	is less than $1/(15k^4)$. As there are at most $K^2 \le 5k^4$ choices for pairs $u,v$ to consider, 
	we obtain by the union bound that $({\rm P}3)$ holds with probability larger than $1-5k^4/(15k^4)=2/3$.
\end{proof}

\noindent The proof of Theorem \ref{t:main} follows now from the following lemma.
\begin{lemma}\label{l:5}
	Let $A$ be a subset of $V(T)$ for which $({\rm P}1)$, $({\rm P}2)$, $({\rm P}3)$ hold. Then, there is a copy of $H_k$
	in $T$ whose base vertices are in $A$.
\end{lemma}
\begin{proof}
	Let $A^* \subseteq A$ be those vertices $u$ of $A$ for which $c(u,v) \le k^{0.8}$ for some
	$v \in A$. Since $({\rm P}2)$ holds, we have that $|A^*| \le 2k^{0.8}$ as there are at most
	$k^{0.8}$ (in fact, at most $k^{0.8}-k^{0.8}/(32k^{0.1})$) pairs $u,v$ of distinct vertices of $A$ with $c(u,v) \le k^{0.8}$. Moreover, since $({\rm P}1)$ holds, we have that $|A| \ge k+2k^{0.8}$ and
	so there is a subset $A' \subseteq A \setminus A^*$ with $|A'|=k$ vertices.
	We will prove that there is an $H_k$ copy in $T$ whose set of base vertices is $A'$.

	Consider an ordering $A'=\{v_1,\ldots,v_k\}$ satisfying that $d^+(v_i) \ge d^+(v_j)$ for all $1 \le i < j \le k$. 
Let also $S$ be a total ordering of the $\binom{k}{2}$ pairs $\{v_i,v_j\}$ with $1 \le i < j \le k$
	where $\{v_i,v_j\}$ precedes $\{v_{i'},v_{j'}\}$ in $S$ implies that
	$c(v_i,v_j) \le c(v_{i'},v_{j'})$.
	Let us now show that $A'$ together with the ordering $S$ of the pairs satisfy the conditions of Lemma~\ref{lem:greedy}, i.e. that for the $\ell$-th pair $\{v_i,v_j\}$ in $S$, it holds that $|N^+(v_i)\cap N^-(v_j)\setminus A|\geq \ell$. This would give the desired copy of $H_k$ with base set $A'$, thus completing the proof.

	Consider the $\ell$-th element of $S$, and suppose it is $\{v_i,v_j\}$ where
	$i < j$. Now, suppose that $c(v_i,v_j)=t$. First, observe that $t \ge k^{0.8}$ as we have already removed $A^*$. By the definition of $S$, we have that $q_t \ge \ell$. But on the other hand, since $({\rm P}2)$ holds, we must have $q_t \le t-t/(32k^{0.1})$. We therefore have that $\ell \le t-t/(32k^{0.1})$ which implies
	$$
	c(u,v) = t \ge \ell + \frac{\ell}{32k^{0.1}}\;.
	$$
	Consider now the set $N^+(v_i) \cap N^-(v_j)$ and let $r := |N^+(v_i) \cap N^-(v_j)|$. Since $d^+(v_i) \ge d^+(v_j)$, we have that
	$r \ge c(v_i,v_j)-1=t-1 \ge k^{0.8}-1 \ge k^{0.7}$.
	Since $({\rm P}3)$ holds, we have that
	$|N^+(u) \cap N^-(v) \cap A| \le r/(2k) + r^{3/4} -1$.
	We therefore have that
	\begin{align*}
		|(N^+(u) \cap N^-(v)) \setminus A| & \ge r - (r/(2k) + r^{3/4} -1) = r\left(1-\frac{1}{2k}-\frac{1}{r^{1/4}}\right)	+ 1 \ge t\left(1-\frac{1}{2k}-\frac{1}{k^{0.2}}\right)\\
		& \ge t\left(1-\frac{2}{k^{0.2}}\right) \ge \left(\ell + \frac{\ell}{32k^{0.1}}\right)\left(1-\frac{2}{k^{0.2}}\right) \ge \ell\;.
	\end{align*}
This completes the proof.
\end{proof}
\section{Concluding remarks}
In this paper we confirmed the conjecture of Gir\~ao, Snyder and Popielarz stating that the oriented Ramsey number of the $1$-subdivision of the transitive tournament is linear. In particular, we show that the necessary size of a tournament which forces such a $1$-subdivision is larger by at most a factor of $4+o(1)$ than the trivial lower bound of $\binom{k}{2} + k$, which can be obtained by noting that this is precisely the number vertices in the $1$-subdivision.
In turn, our proof cannot give a tight bound - this is because the bound on $c(u,v)$ in Lemma~\ref{l:0} is tight, i.e. there exist tournaments (namely, those which are doubly-regular, see \cite{doubreg}) for which we know that $c(u,v) = \frac{n-3}{4}$ for every pair of vertices $u,v$. Therefore, if $n < 2k^2$, we cannot use Lemma~\ref{lem:greedy} as an embedding strategy since for each $n/4 \leq t \leq \binom{k}{2}$, there will not exist pairs $u,v$ with $|N^+(u) \cap N^-(v)| \geq t$. Despite this, it is natural to ask whether indeed a 'spanning' behaviour for this problem is true at least in an asymptotic form, i.e., if the oriented Ramsey number is $k^2/2 + o(k^2)$.

\end{document}